\documentclass[a4paper,11pt,leqno]{article}

\usepackage{amsmath}
\usepackage{amssymb}
\usepackage{amsthm}
\usepackage{epsfig}

\def\RR{{\bf R}}
\def\ZZ{{\bf Z}}
\def\bigmid {\ \left|{\Large \strut}\right.}
\def\11{{\bf 1}}

\numberwithin{equation}{section}

\newcommand{\val}{\mathop{\rm val}}

\newtheorem{Thm}{Theorem}[section]
\newtheorem{Prop}[Thm]{Proposition}
\newtheorem{Lem}[Thm]{Lemma}
\newtheorem{Cor}[Thm]{Corollary}
\theoremstyle{definition}

\newtheorem{Rem}[Thm]{Remark}

\title{On tight spans and tropical polytopes for directed distances}
\author{Hiroshi HIRAI \\
        Research Institute for Mathematical Sciences, \\
        Kyoto University, Kyoto 606-8502, Japan  \\
        \texttt{\normalsize hirai@kurims.kyoto-u.ac.jp}
        \\ \\
        Shungo KOICHI \\
        Department of Systems Design and Engineering,  \\
        Nanzan University, Seto 489-0863, Japan  \\
        \texttt{\normalsize shungo@nanzan-u.ac.jp}}
\date{April 2010}

\begin{document}

\maketitle

\begin{abstract}
An extension $(V,d)$ of a metric space $(S,\mu)$ 
is a metric space with $S \subseteq V$ and $d|_S = \mu$, 
and is said to be tight if 
there is no other extension $(V,d')$ of $(S,\mu)$ with $d' \leq d$. 
Isbell and Dress
independently found that every tight extension 
is isometrically embedded into 
a certain metrized polyhedral complex 
associated with $(S,\mu)$, called the tight span.
This paper develops an analogous theory 
for directed metrics, 
which are ``not necessarily symmetric''
distance functions satisfying the triangle inequality. 
We introduce a directed version of the tight span
and show that it has 
such a universal embedding property 
for tight extensions.
Also we newly introduce another natural class of extensions, 
called cyclically tight extensions, 
and show that (a fiber of) the tropical polytope, 
introduced by Develin and Sturmfels, 
has a universal embedding property 
for cyclically tight extensions.
As an application, 
we prove the following directed version of tree metric theorem:
directed metric $\mu$ is a directed tree 
metric if and only if 
the tropical rank of $-\mu$ is at most two.
Also we describe how tight spans and tropical polytopes 
are applied to the study in multicommodity flows 
in directed networks.
\end{abstract}

\section{Introduction}
\label{sec:Introduction}
An {\em extension} $(V,d)$ of a metric space $(S,\mu)$ 
is a metric space with $S \subseteq V$ and $d|_S = \mu$, 
and is said to be {\em tight} if 
there is no other extension $(V,d')$ of $(S,\mu)$ with $d' \leq d$. 
Isbell~\cite{Isbell64} and Dress~\cite{Dress84} 
independently found that
there exists a {\em universal} tight extension 
with property that it contains every tight extension as a subspace.
The universal tight extension is realized 
by the set of minimal elements 
of the polyhedron
\[
 \{ p \in \RR^S \mid p(s) + p(t) \geq \mu(s,t)\ (s,t \in S) \}
\]
in $(\RR^S,l_{\infty})$, 
which is called the {\em injective hull} or
the {\em tight span}.
The tight span reflects, in geometric way, 
several hidden combinatorial properties of $\mu$.
For example, $\mu$ is a tree metric 
if and only if its tight span is a tree (one-dimensional contractible complex); 
see \cite[Theorem 8]{Dress84}.
This property reinterprets 
the classical {\em tree metric theorem}~\cite{Buneman74, Pereira69, Zareckii65}, 
and has a significance 
in the phylogenetic tree reconstruction 
in the biological sciences.
Tight spans also arise in quite different contexts, 
including the {\em $k$-server problem} in computer sciences~\cite{CL94} 
and the {\em multicommodity flows problems} 
in combinatorial optimization~\cite{Hirai09, Kar98EJC}.
In applications to both areas, 
the geometry of tight spans and 
the above-mentioned universality play crucial roles.

The goal of this paper is to develop 
an analogous theory of tight spans for directed distances and metrics.
A {\em directed distance} $d$ on a set $V$ is
a nonnegative real-valued 
function $d: V \times V \to \RR_+$ 
having zero diagonals $d(x,x) = 0$ $(x \in V)$.
A directed distance $d$ is said to be 
a {\em directed metric} if it satisfies 
the triangle inequalities $d(x,y) + d(y,z) \geq d(x,z)$ $(x,y,z \in V)$.
In both definitions, 
we do not impose the symmetric condition $d(x,y) = d(y,z)$.
A pair $(V,d)$ of a set $V$ and its directed metric $d$ 
is called a {\em directed metric space}.
We can define the concept of tight extensions in the same way.
The motivation for our theory comes from the following natural questions:
\begin{itemize}
\item {\em Does there exist the universal tight extension?} 
\item {\em If exists, does its geometry characterize 
a class of combinatorial distances, 
related to directed versions of tree metrics?}
\item {\em Are there reasonable applications to 
the study of multicommodity flows in directed networks?}
\end{itemize}
The main contribution of this paper
answers these questions affirmatively. 
In particular a directed analogue of the tight span does exist.
Furthermore we newly 
introduce another natural class of extensions, 
called {\em cyclically tight extensions.}
Interestingly, it has an unexpected connection to
the theory of {\em tropical polytopes}, 
developed by Develin and Sturmfels~\cite{DS04}.
In fact, 
we show that the tropical polytope 
generated by the distance matrix of $\mu$
has the universal embedding property for cyclically tight extensions.
In Section~\ref{sec:tightspans},
for a directed distance $\mu$, 
we define a directed analogue $T_{\mu}$ of the tight span 
(the {\em directed tight span}) as 
the set of minimal points in 
an unbounded polyhedron $P_{\mu}$ associated with $\mu$.
The tropical polytope is the projection $\bar Q_{\mu}$ of 
the set $Q_{\mu}$ of minimal points in 
another polyhedron associated with $\mu$.
We introduce 
a well-behaved {\em section}, 
a subset in $Q_{\mu}$ projected bijectively into $\bar Q_{\mu}$, 
called a {\em balanced section}.
We endow $T_{\mu}$, $Q_{\mu}$, and any balanced section $R$ 
with a directed analogue of the $l_{\infty}$-metric.
We investigate the behaviors of 
these directed metric spaces, 
which is placed in the main body of Section~\ref{sec:tightspans}. 
Then we prove the universal embedding properties 
that every tight extension of 
a metric $\mu$ is 
isometrically embedded into the tight span $T_{\mu}$, 
and every cyclically tight extension is isometrically embedded 
into a balanced section of $Q_{\mu}$ (Theorem~\ref{thm:extension}).

In Section~\ref{sec:dim1},
as an application of the embedding property
and the dimension criteria, 
we prove that some classes of distances 
realized by {\em oriented trees}
can be characterized by one-dimensionality 
of tight spans and tropical polytopes (Theorems~\ref{thm:path} and \ref{thm:tree}).
In particular, we prove the following
directed version of the tree metric theorem:
{\em a directed metric $\mu$ is a directed tree metric if and only if 
the tropical rank of $-\mu$ is at most $2$} (Corollary~\ref{cor:treemetric}). 
This complements the argument in \cite[Section 5]{DS04}.

In Section~\ref{sec:multiflows}, 
we briefly sketch 
how tight spans and tropical polytopes are applied to 
the study of directed multicommodity flow problems.
We prove that the linear programming dual 
to multicommodity flow problems reduces 
to {\em facility location problems} on tight spans and tropical polytopes 
(Theorems~\ref{thm:Tdual} and \ref{thm:Qdual}).
Further study 
of this approach
will be given in the next paper~\cite{HK}.

\paragraph{Notation.}
The sets of real numbers and nonnegative real numbers
are denoted  by $\RR$ and $\RR_{+}$, respectively.
For a set $X$, 
the sets of functions 
from $X$ to $\RR$ and from $X$ to $\RR_+$ are 
denoted by $\RR^{X}$ and $\RR_{+}^X$, respectively.
For a subset $Y \subseteq X$,
the characteristic function $\11_{Y} \in \RR^{X}$ 
is defined by 
$\11_Y(x) = 1$ for $x \in Y$
and $\11_{Y}(x) = 0$ for $x \notin Y$.
We particularly 
denote by $\11$ 
the all-one function $\11_X$ in $\RR^{X}$.
The singleton set $\{ s \}$ is often denoted by $s$, 
such as $\11_s$ instead of $\11_{\{s\}}$. 
For $p,q \in \RR^{X}$, $p \leq q$ means 
$p(x) \leq q(x)$ for each $x \in X$, and $p < q$ means 
$p(x) < q(x)$ for each $x \in X$.
For $p \in \RR^{X}$, $(p)_{+}$ is a point in $\RR^X$ defined by
$((p)_{+})(x) = \max\{ p(x),0 \}$ for $x \in X$.
For a set $P$ in $\RR^X$, 
a point $p$ in $P$ is said to be {\em minimal} 
if there is no other point $q \in P \setminus p$ with $q \leq p$.

In this paper, a directed metric (distance) 
is often simply called a metric (distance).
A metric (distance) in the ordinary sense is particularly 
called 
an {\em undirected metric} ({\em undirected distance}).
We use the terminology of undirected metrics in an analogous way
(we do not know any reference including a systematic 
 treatment for directed metrics).
For two directed metric spaces $(V,d)$ and $(V',d')$,
an {\em isometric embedding} from $(V,d)$ to $(V',d')$
is a map $\rho: V \to V'$ satisfying $d'(\rho(x),\rho(y)) = d(x,y)$
for all (ordered) pairs $x,y \in V$.
For a directed metric $d$ on $V$, and two subsets $A,B \subseteq V$, 
let $d(A,B)$ denote the minimum distance from $A$ to $B$:
\[
d(A,B) = \inf \{ d(x,y) \mid (x,y) \in A \times B\}.
\]
In our theory, the following directed metric 
$D_{\infty}^+$ on $\RR^X$ is particularly important: 
\[
D_{\infty}^+(p,q) = \|(q- p)_+ \|_{\infty} 
\quad ( = \max_{x \in X} (q - p)_+ (x) ) \quad (p,q \in \RR^X).
\]

For an directed or undirected graph $G$, 
its vertex set and edge set 
are denoted by $VG$ and $EG$, respectively.
If directed, an edge with tail $x$ and head $y$ 
is denoted by $xy$.
If undirected, we do not distinguish $xy$ and $yx$.

\section{Tight spans and tropical polytopes}
\label{sec:tightspans}

In this section, we introduce and study
the tight span and the tropical polytope
of a directed distance.
Let $\mu$ be a directed distance on a finite set $S$.
Let $S^c$ and $S^r$ be copies of $S$.
For an element $s \in S$, 
the corresponding elements 
in $S^c$ and $S^r$ are denoted by $s^c$ and $s^r$,
respectively.
We denote $S^c \cup S^r$ by $S^{cr}$.
For a point $p \in \RR^{S^{cr}}$, 
the restrictions of $p$ to $S^c$ and $S^r$
are denoted by $p^c$ and $p^r$, respectively, i.e., $p = (p^c, p^r)$.
We define the following polyhedral sets:
\begin{eqnarray*}
{\mit\Pi}_{\mu} & = & \{  p \in \RR^{S^{cr}}
\mid p(s^c) + p(t^r) \geq \mu(s,t) \; (s,t \in S) \}. \\
P_{\mu} & = &  {\mit\Pi}_\mu \cap \RR_{+}^{S^{cr}}. \\
T_{\mu} & = & \mbox{the set of minimal elements of } P_{\mu}. \\
Q_{\mu} & = &  \mbox{the set of minimal elements of } {\mit\Pi}_{\mu}. \\
Q^+_{\mu} &=& 
Q_\mu \cap \RR^{S^{cr}}_{+}.
\end{eqnarray*}
We call $T_{\mu}$ the {\em directed tight span}, 
or simply, the {\em tight span}. 
The polyhedron ${\mit\Pi}_\mu$ 
has the linearity space $(\11,-\11)\RR$. 
The natural projection of vector $v \in \RR^{S^{cr}}$ 
to $\RR^{S^{cr}}/(\11,-\11)\RR$
is denoted by $\bar v$.
The projection $\bar Q_{\mu}$ of $Q_{\mu}$ 
coincides with the {\em tropical polytope} 
generated by $S$ by $S$ matrix 
$( - \mu(s,t) \mid s,t \in S)$; 
see Develin and Sturmfels~\cite{DS04}.
We note the following relation:
\[
\begin{array}{ccccccc}
Q_\mu & \supseteq & Q^+_{\mu} & \subseteq & T_{\mu} & \subseteq & P_{\mu} \\
\downarrow & \swarrow &    &  &  & & \\
\bar Q_\mu & & & & & & 
\end{array}
\]
Here the arrow means the projection.
In the inclusions, $Q^+_{\mu}$ is a subcomplex of $T_{\mu}$, 
and $T_{\mu}$ is a subcomplex of $P_{\mu}$.

A subset $R \subseteq Q_{\mu}$ is called 
a {\em section} if the projection $p \in R \mapsto \bar p$ is bijective.
A subset $R \subseteq \RR^{S^{cr}}$ is 
said to be {\em balanced} if there is no pair $p,q$ of points in $R$
such that $p^c < q^c$ or $p^r < q^r$.
We are mainly interested in balanced sections.
We metrize these subsets as follows.
Let $D_{\infty}$ be a directed metric on $\RR^{S^{cr}}$ defined as
\[
D_{\infty}(p,q) 
 =  \max\{ D_{\infty}^{+} (p^c,q^c),\; D_{\infty}^{+}(q^r, p^r) \}  \quad (p,q \in \RR^{S^{cr}}). \nonumber 
\]
As subspaces of $(\RR^{S^{cr}}, D_{\infty})$, 
we obtain directed metric spaces 
$(T_{\mu},D_{\infty})$, $(Q_{\mu}, D_{\infty})$,  $(Q^+_{\mu}, D_{\infty})$, 
and $(R, D_{\infty})$ for any balanced section $R \subseteq Q_{\mu}$.
Figure~\ref{fig:Qmu} illustrates those subsets 
for all-one distance $\mu$ on a 3-set $S = \{s,t,u\}$.
In this case, $T_{\mu} = Q_{\mu}^+$ (accidentally), 
$Q_{\mu}$ is the union of three infinite strips with a common side, 
$Q_{\mu}^+$ is a {\em folder} consisting of three triangles, 
and $\bar Q_{\mu}$ is a star of three leaves.
Here $\mu_s$ is a vector in $\RR^{S^{cr}}$ 
consisting of the $s$-th column and the $s$-th row of $\mu$; 
see Section~\ref{subsec:embedding} for definition.
\begin{figure}
\center
\epsfig{file=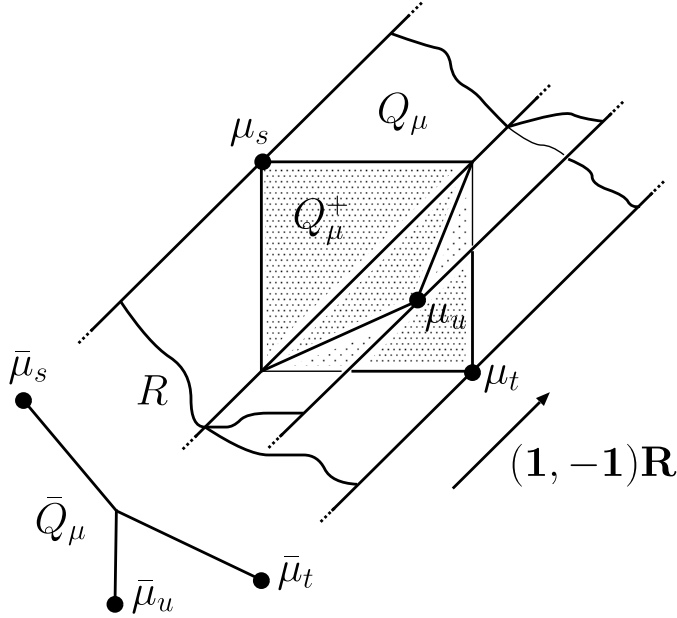,width=0.6\hsize}
\caption{$Q_\mu$, $Q_\mu^+$, and $\bar Q_\mu$}
\label{fig:Qmu}
\end{figure}

The rest of this section is 
organized as follows.
In the next Section~\ref{subsec:basic}, 
we present basic properties of these polyhedral sets.
In Section~\ref{subsec:nonexpansive}, 
we prove the existence of 
certain {\em nonexpansive retractions} 
among them.
This is the most substantial part of this section.
In Section~\ref{subsec:embedding}, 
we show that $T_{\mu}, Q_{\mu}, Q_{\mu}^+$, and any balanced section
are {\em geodesic} in the {\em directed} sense.
Also we show that distance $\mu$ 
is {\em embedded into} them.
In Section~\ref{subsec:extension},
we introduce notions of tight extensions and 
cyclically tight extensions of a directed metric space $(S,\mu)$,  
and prove that $T_{\mu}$ and $Q_{\mu}^+$ 
have the universal embedding properties 
for tight extensions and cyclically tight extensions, respectively. 
In Section~\ref{subsec:dimension}, we present dimension criteria 
for $T_{\mu}$ and $\bar Q_{\mu}$.
In the sequel, $\mu$ is supposed to be 
a distance on a finite set $S$, without noted.

\subsection{Basic properties}\label{subsec:basic}
Here
we briefly summarize some fundamental properties of 
$T_{\mu}$, $Q_{\mu}$, and $Q_{\mu}^+$. 
For a point $p \in \RR^{S^{cr}}$,
we denote by $K(p)$ the
bipartite graph with bipartition $(S^c,S^r)$ and
edge set 
$EK(p) = \{ s^ct^r \mid s^c \in S^c, t^r \in S^r, p(s^c) + p(t^r) = \mu(s,t) \}$.
The minimality of $p$ 
can be rephrased in terms of $K(p)$ as follows.
\begin{Lem}\label{lem:minimality}
\begin{itemize}
\item[{\rm (1)}]
A point $p$ in $P_{\mu}$ belongs to
$T_{\mu}$
if and only if
$K(p)$ has no isolated vertices $u$ with $p(u) > 0$.
\item[{\rm (2)}]
A point $p$ in ${\mit\Pi}_{\mu}$ belongs to $Q_{\mu}$
if and only if $K(p)$ has no isolated vertices.
\end{itemize}
\end{Lem}
The next lemma says that
the projections of $T_{\mu}$ and of $Q_{\mu}$ 
to $\RR^{S^c}$ (or $\RR^{S^r}$) are isometries. 
So we can consider $T_\mu$, $Q_\mu$, and $Q_{\mu}^+$ 
in $(\RR^{S^c}, D_{\infty}^{+})$.
\begin{Lem}\label{lem:mutual}
\begin{itemize}
\item[{\rm (1)}]
$D_{\infty}(p,q) = D_{\infty}^+(p^c, q^c) = D_{\infty}^{+}(q^r, p^r)$ 
for $p,q \in T_{\mu}$ and for $p,q \in Q_{\mu}$.
\item[{\rm (2)}]
For $p,q \in Q_{\mu}$, $p^c < q^c$ if and only if $p^r > q^r$. 
\end{itemize}
\end{Lem}
\begin{proof}
(1). 
Suppose $D_{\infty}^+(p^c,q^c) = q(s^c) - p(s^c) > 0$
for $s^c \in S^c$.
By Lemma~\ref{lem:minimality},
there is $t^r \in S^r$ such that $q(s^c) + q(t^r) = \mu(s,t)$.
Hence, 
$q(s^c) - p(s^c)
= \mu(s,t) - q(t^r) - p(s^c)
\leq p(t^r) - q(t^r) \leq D_{\infty}^+(q^r,p^r)$.
Thus $D_{\infty}^+(p^c,q^c) \leq D_{\infty}^+(q^r,p^r)$.
Similarly $D_{\infty}^+(p^c,q^c) \geq D_{\infty}^+(q^r,p^r)$.

(2). Suppose $p^c < q^c$. 
If $p(s^r) \leq q(s^r)$ for some $s^r \in S^r$, 
then there is $t^c \in S^c$ with $q(t^c) + q(s^r) = \mu(t,s)$, 
and $p(t^c) + p(s^r) < q(t^c) + q(s^r) = \mu(t,s)$; 
a contradiction to $p \in Q_{\mu}$.
\end{proof}
Next we study the local structure of $T_{\mu}$.
For a point $p \in T_{\mu}$, 
let $F(p)$ denote the face of $T_{\mu}$ that contains $p$ 
in its relative interior.
Suppose that $K(p)$ has 
$k$ (bipartite) components having 
no vertices $u \in S^{cr}$ with $p(u) = 0$.
Let $(A^c_1,B^r_1),\dots,(A^c_k,B^r_k)$ 
be bipartitions of these components.
Then the affine span of $F(p)$ is given by 
\begin{equation}\label{eqn:coordinate}
p + \sum_{i=1}^{k} x_i (\11_{A^c_i}, -\11_{B^r_i}) 
\quad (x_1,\ldots,x_k \in \RR).
\end{equation}
Therefore the dimension of $F(p)$ is given as follows; 
compare \cite[Proposition 17]{DS04}.
\begin{Lem}\label{lem:dim}
For a point $p \in T_{\mu}$, 
the dimension of $F(p)$ 
is equal to the number of connected components of $K(p)$
having no vertex $s$ with $p(s) = 0$.
\end{Lem}
We end this subsection with a seemingly obvious lemma. 
\begin{Lem}\label{lem:section}
For any balanced set $B \subseteq Q_\mu$, 
there exists a balanced section $R \subseteq Q_\mu$
containing $B$.
If $B \subseteq Q^+_\mu$, 
then there exists a balanced section $R \subseteq Q^+_{\mu}$ containing $B$.
\end{Lem}
\begin{proof}
Let $B$ be a balanced set in $Q_{\mu}$.
By Zorn's lemma, 
we can take an inclusion-maximal balanced set $R \subseteq Q_\mu$
containing $B$.
We show that $R$ is a section.
Suppose not.
Then there is $p_0 \in Q_\mu$ with 
$\{p_0 + (\11,-\11) \RR\} \cap B = \emptyset$.
Let $L =  p_0 + (\11,-\11) \RR$.
For each point $q \in \RR^{S^{cr}}$, 
let $L_q$ be the set of points $p$ in $L$ such that
$\{p,q\}$ is balanced.
Then $L_q$ is a closed segment in $L$; see below.
We claim:
\begin{itemize}
\item[($*$)] For every balanced pair $q,q' \in \RR^{S^{cr}}$, 
the intersection $L_q \cap L_{q'}$ is nonempty.
\end{itemize}
Suppose true. 
By Helly property of segments on a line,
the intersection 
$\bigcap \{ L_q \mid q \in B\}$ is nonempty.
Take any point $q \in \bigcap \{ L_q  \mid q \in B \}$.
Then $\{ q \} \cup B$ is balanced, 
which contradicts the maximality of $B$.

So we show ($*$). 
In the following, 
for $p \in \RR^X$ we simply denote $\min_{x \in X} p(x)$ and 
$\max_{x \in X} p(x)$ by $\min_{} p$ and $\max_{} p$, respectively.
$L_q$ is indeed a segment given by 
$\{ p_0 + \alpha(\11,-\11) \mid  l_q \leq \alpha \leq u_q \}$, 
where $l_q := \min_{} (q -p_0)^c = \min_{} (p_0 - q)^r$ and 
$u_q :=  \max_{} (q - p_0)^c = \max_{} (p_0 - q)^r$.
This follows from: $\{q, p_0+ \alpha (\11,-\11)\}$ is balanced 
$\Leftrightarrow$ $\min_{} ({p_0}^c + \alpha \11 - q^c) \leq 0$ and 
$\min_{} (q^c - {p_0}^c - \alpha \11) \leq 0$ 
$\Leftrightarrow$ $\min_{} ({p_0}^r - \alpha \11 - q^r) \leq 0$ and 
$\min_{} (q^r - {p_0}^r + \alpha \11) \leq 0$ (Lemma~\ref{lem:mutual}~(2)).
Therefore, if $L_q \cap L_q' = \emptyset$, 
then either $\max_{} (q - p_0)^c < \min_{} (q' - p_0)^c$ or 
$\min_{} (q - p_0)^c > \max_{} (q' - p_0)^c$ holds, 
and this implies $q^c < (q')^c$ or $q^c > (q')^c$, 
i.e., $\{q,q'\}$ is not balanced.

For the latter part, by the same Helly idea, 
it suffices to show that 
$L \cap Q_{\mu}^+$ is nonempty, and
$L_q \cap Q_{\mu}^+$ 
is nonempty if $q \in Q_{\mu}^+$.
As $L \cap Q_{\mu}^+ (= L \cap \RR^{S^{cr}}_+)$ is given 
by $\{p_0 + \alpha(\11,-\11) \mid 
\max_{} -{p_0}^c \leq \alpha \leq \min_{} {p_0}^r \}$, 
it is nonempty by $p_0(s^c) + p_0(t^r) \geq \mu(s,t) \geq 0$.
Moreover, if $L_q \cap Q_{\mu}^+ = \emptyset$, then
either $\min_{} {p_0}^r < \min_{} (p_0 - q)^r$ 
or $\max_{} (q - p_0)^c < \max_{} -{p_0}^c$ holds;
both cases are impossible by $q \geq 0$.
\end{proof}

\subsection{Nonexpansive retractions}\label{subsec:nonexpansive}
The goal of this subsection is 
to show the existence of several {\em nonexpansive retractions} 
from $P_{\mu}$ to $T_{\mu}$, from $T_{\mu}$ to $Q^+_{\mu}$, 
and from $Q_{\mu}$ to any balanced section.
At the first glance, 
they are seemingly technical lemmas. 
However it will turn out that
many important properties can be easily derived from them.
The first retraction lemma is about $P_{\mu} \rightarrow T_{\mu}$, 
which is a natural extension of \cite[(1.9)]{Dress84}.
\begin{Lem}\label{lem:nonexpansive}
There exists a map $\phi:P_{\mu} \to T_{\mu}$ such that
\begin{itemize}
\item[{\rm (1)}] $\phi(p) \leq p$ for $p \in P_{\mu}$ (and thus
	     $\phi(p) = p$ for $p \in T_{\mu}$), and
\item[{\rm (2)}] $D_{\infty}(\phi(p), \phi(q)) \leq D_{\infty}(p,q)$ for $p,q \in P_{\mu}$. 
\end{itemize}
\end{Lem}
The property (2) says that 
this map $\phi$ does not expand 
the distance of any pair of points in $P_{\mu}$.
A map between metric spaces with this property 
is said to be {\em nonexpansive}. 

In a directed metric space $(V,d)$, 
there is another notion of nonexpansiveness.
By a {\em cycle} in $V$
we mean a cyclic permutation $(p_1,p_2,\ldots,p_m)$
of (possibly repeated) points $p_1,p_2,\ldots,p_m \in V$. 
For a cycle $C = (p_1,p_2,\ldots,p_m)$, 
the {\em length} $d(C)$ is defined by
\[
d(C) = d(p_1,p_2) + d(p_2,p_3) + \ldots + d(p_{m-1},p_{m}) + d(p_{m},p_{1}). 
\]
For a map $\varphi: V \to V$, 
let $\varphi(C)$ denote the cycle $(\varphi(p_1), \varphi(p_2),\ldots, \varphi(p_m))$.
Then 
there is a retraction 
from $T_{\mu}$ to $Q^+_{\mu}$ not expanding the length of every cycle.
\begin{Lem}\label{lem:cyclic_nonexpansive}
There exists a map $\varphi : T_{\mu} \to Q^+_{\mu}$ such that 
\begin{enumerate}
 \item[{\rm (1)}] $\varphi(p) = p$ for $p \in Q^+_{\mu}$, and
 \item[{\rm (2)}] $D_{\infty}(\varphi(C)) \leq D_{\infty}(C)$ for every cycle $C$ in $T_{\mu}$.
\end{enumerate}
\end{Lem}
We call a map with property (2) {\em cyclically nonexpansive}.
Any nonexpansive map is cyclically nonexpansive, 
and the converse is not true in general.
Two notions coincide in undirected metric spaces.

Let $R \subseteq Q_{\mu}$ be a balanced section. 
For any point $p \in Q_{\mu}$, 
there uniquely exists $q \in R$ with $\bar p = \bar q$.
From this correspondence, 
we obtain a retraction $\varphi_R: Q_{\mu} \to R$. 
Again $\varphi_R$ is cyclically nonexpansive.
\begin{Lem}\label{lem:balanced}
Let $R \subseteq Q_{\mu}$ be a balanced section.
\begin{itemize}
\item[{\rm (1)}] $D_{\infty}(\varphi_R(C)) \leq D_{\infty}(C)$ for every cycle $C$ in $Q_{\mu}$, and
\item[{\rm (2)}] subset $U \subseteq Q_{\mu}$ is balanced 
if and only if $D_{\infty}(\varphi_R(C)) = D_{\infty}(C)$ holds 
for every cycle $C$ in $U$. 
\end{itemize}
\end{Lem}
We give a remark concerning the continuity of these maps.
Recall that any nonexpansive map between undirected metric spaces
is always continuous 
(with respect to the metric topology).
\begin{Rem}\label{rem:continuity}
One can see 
that map $p \mapsto \max \{(p^c)_+, (- p^r)_+\} + \max \{(- p^c)_+, (p^r)_+ \}$ 
defines a norm $\| \cdot \|$ in $\RR^{S^{cr}}$.
The corresponding distance $\| p - q \|$ equals 
to $D_{\infty}(p,q) + D_{\infty}(q,p)$.
Therefore any cyclically nonexpansive retraction 
is nonexpansive 
in this metric, 
and is continuous with 
respect to the topology induced by this norm.
Since 
any norm in $\RR^n$ defines the same topology, 
it is continuous 
with respect to the Euclidean topology.
As a consequence,  
all $T_{\mu}, Q_{\mu}^{+}, Q_{\mu}, \bar Q_{\mu}$ 
are contractible (the contractibility of tropical polytopes was shown
in \cite{DS04}).
In particular, any balanced section is a continuous section.
\end{Rem}
The rest of this subsection is devoted 
to proving three lemmas.
For a nonzero vector $v \in \RR^{S^{cr}}$ and 
a nonnegative real $\alpha \geq 0$, 
let $\phi[v, \alpha]: P_{\mu} \to P_{\mu}$ be the map defined by
\[
\phi[v,\alpha](p) = p + v 
\max \{ \epsilon \in [0,\alpha] \mid p + \epsilon v \in P_{\mu}\} \quad (p \in P_{\mu}).
\] 
This map moves each point in $P_{\mu}$ toward direction $v$ 
in same speed
as much as possible within time $\alpha$.
Let $\phi[v] := \lim_{\alpha \rightarrow \infty} \phi[v,\alpha]$, 
which is well-defined if $v$ has a negative component.
By definition, $\phi[v]$ is identity 
on the set of points $p$
such that the tangent cone of $P_{\mu}$ at $p$
does not contain $v$.
We will construct required retractions as 
the composition of $\phi[v]$ for several vectors $v$.
\paragraph{Proof of Lemma~\ref{lem:nonexpansive}.}
For $s \in S$, 
let $\phi_s := \phi[-\11_{s^c}] \circ \phi[-\11_{s^r}]$.  
Let $S = \{s_1,s_2,\ldots,s_n\}$ and 
let $\phi := \phi_{s_1} \circ \phi_{s_2} \circ \cdots \circ \phi_{s_n}$.
We remark that 
the map $\phi[-\11_u]$ 
decreases the $u$-th component of each point $p$ 
as much as possible, or equivalently, 
until ($*$) $p(u) = 0$ or
there appears an edge covering $u$ in $K(p)$ (Lemma~\ref{lem:minimality}~(1)).  
From this, we see that $\phi$ is a retraction from $P_\mu$ to $T_{\mu}$
with property (1).
So it suffices to show that $\phi_s$ is nonexpansive for $s \in S$.
Take $p,q \in P_\mu$. 
It suffices to compare their $s^c$- and $s^r$-components.
By ($*$) we have
$\phi_{s}(q)(s^c) = 
\max\{0, \mu(s,u) - q(u^r) \mid u \in S \} =
\max\{0, (\mu(s,u) - p(u^r)) + (p(u^r) - q(u^r)) \mid u \in S \} \leq 
\phi_{s}(p)(s^c) + \| (p^r - q^r)_{+}  \|_{\infty}$
and
$\phi_{s}(p)(s^r) = 
\max\{0, \mu(u,s) - p(u^c) \mid u \in S \} =
\max\{0, (\mu(u,s) - q(u^c)) + (q(u^c) - p(u^c)) \mid u \in S \} \leq 
\phi_{s}(q)(s^r) + \| (q^c - p^c)_{+}  \|_{\infty}$.
Thus $(\phi_{s}(q)(s^c) - \phi_s(p) (s^c))_+ \leq \| (p^r - q^r)_{+}  \|_{\infty}$
and $(\phi_{s}(p)(s^r) - \phi_s(q) (s^r))_+ \leq \| (q^c - p^c)_{+}  \|_{\infty}$.
Consequently we have 
$D_{\infty}(\phi_s(p),\phi_s(q)) \leq D_{\infty}(p,q)$.
This proof method 
is a direct adaptation of that of \cite[(1.9)]{Dress84}.

\paragraph{Proof of Lemma~\ref{lem:cyclic_nonexpansive}.}
By Lemma~\ref{lem:minimality}, 
any point $p \in T_{\mu} \setminus Q^+_{\mu}$ 
has isolated vertices in $K(p)$.
Since $\mu \geq 0$, 
the set $I$ of isolated vertices in $K(p)$ belongs to either $S^c$ or $S^r$.
If $I \subseteq S^c$, 
then $p + \epsilon (\11_{S^c \setminus I}, - \11) \in T_\mu$
for small $\epsilon > 0$, and if $I \subseteq S^r$, 
then $p + \epsilon (- \11, \11_{S^r \setminus I}) \in T_\mu$
for small $\epsilon > 0$.
On the other hand, if $p \in Q_{\mu}^+$, 
then $p + \epsilon (\11_{A^c}, - \11) \not \in T_{\mu}$ and
$p + \epsilon (- \11, \11_{A^r}) \not \in T_{\mu}$
for any nonempty proper subset $A \subseteq S$.
Define 
$\varphi^c_{A} = \phi[(\11_{A^c}, - \11)]$ 
and $\varphi^r_A = \phi[(- \11, \11_{A^r})]$ for a nonempty proper subset $A \subseteq S$.
Then both $\varphi^c_{A}$ and $\varphi^r_A$ are identity on
$Q^+_{\mu}$.
Order all nonempty proper subsets $A_1,A_2,\ldots, A_{m}$ $(m = 2^{|S|}-2)$ 
of $S$ such that $A_i \subseteq A_j$ implies $i \leq j$. 
Consider point $p \in T_{\mu}$ whose $K(p)$ 
has nonempty set $I$ of isolated vertices in $S^c$.
Apply all maps $\varphi_{A_i}^c$ ($i=1,2,\ldots, m$) to $p$ 
according to the ordering of subsets $A_i$.
In a step $j$, 
point $p$ moves if and only if ${A_j}^c = S^c \setminus I$.
If moves, then it moves until some $s^c \in I$ is covered by edges.
So the set $I$ of isolated vertices decreases.
By the definition of the ordering, 
there is a step $j'> j$ such that ${A_{j'}}^c = S^c \setminus I$.
Therefore, after the procedure, 
$K(p)$ has no isolated vertex, i.e., $p \in Q_{\mu}^+$.
Motivated by this, 
let us define
$\varphi^c = \varphi^c_{A_m} \circ \varphi^c_{A_{m-1}} \circ \cdots \circ
 \varphi^c_{A_2} \circ \varphi^c_{A_1}$ and 
$\varphi^r = \varphi^r_{A_m} \circ \varphi^r_{A_{m-1}} \circ \cdots \circ 
 \varphi^r_{A_2} \circ \varphi^r_{A_1}$.
By the argument above, $\varphi := \varphi^r \circ \varphi^c$ is a retraction 
from $T_\mu$ to $Q^+_\mu$.
So it suffices to verify that $\varphi^c_{A}$ 
is cyclically nonexpansive for each $A \subseteq S$.
Take any cycle $C = (p_1,p_2,\ldots, p_{n})$ 
in $T_{\mu}$.
For $\alpha \geq 0$, 
let $\varphi^c_{A, \alpha} = \phi[(\11_{A^c}, -\11), \alpha]$.
Define function $h: \RR_{+} \to \RR_{+}$ by
$h(\alpha) := D_{\infty}(\varphi^c_{A, \alpha}(C))$ 
for $\alpha \in \RR_+$, which is piecewise linear continuous.
We show that $h$ is nonincreasing.
It suffices to verify that the right derivative of $h$ 
at $\alpha = 0$ is nonpositive.
Let $i(C)$ be the set of indices $j$ 
such that $\varphi^c_{A, \epsilon}(p_j) \neq p_j$ 
for every $\epsilon > 0$.
For a small $\epsilon^* >0$, 
let $C^{*} = (p^{*}_1,p^{*}_2,\ldots, p^{*}_{n}) 
:= \varphi^c_{A, \epsilon^*}(C)$.
Then $p^*_j = p_j + \epsilon^* (\11_{A^c}, -\11)$ if $j \in i(C)$, 
and $p^*_j = p_j$ otherwise.
We show $D_{\infty}(C^{*}) \leq D_{\infty}(C)$.
We may assume that there are indices $k,l$ 
such that $k-1,l+1 \not \in i(C)$ and $k,k+1,\ldots, l-1,l \in i(C)$,
where the indices are taken modulo $n$.
For $j = k,k+1, \dots, l-1$,
we have $D_\infty(p_j, p_{j+1}) = D_\infty(p^{*}_j, p^{*}_{j+1})$.
Consider
$D_\infty(p^{*}_{l},p^{*}_{l+1}) 
= D_{\infty}(p_{l} + \epsilon^{*} (\11_{A^c},-\11), p_{l+1})$. 
Since $\varphi_{A}^c(p_{l+1}) = p_{l+1}$,
there is an edge $s^{c}t^{r}$
in $K(p_{l+1})$ 
connecting $S^c \setminus A^c$ and $S^r$.
Then
$\epsilon^* \leq p_{l}(s^c) + p_{l}(t^r) - \mu(s,t) 
 = p_{l}(s^c) + p_{l}(t^r) - p_{l+1}(s^c) - p_{l+1}(t^r)
 = p_{l}(t^r) - p_{l+1}(t^r) - p_{l+1}(s^c)$,
where we use $p_{l+1}(s^c) + p_{l+1}(t^r) = \mu(s,t)$
by $s^ct^r \in EK(p_{l+1})$
and $p_{l}(s^c) = 0$ (since $s^c$ is isolated in $K(p_l)$).
By Lemma~\ref{lem:mutual} and 
$p_{l}(t^r) - p_{l+1}(t^r) \geq \epsilon^*$, 
we have
\begin{equation*}
D_{\infty}(p^{*}_{l},p^{*}_{l+1})
= \max \{ (p_{l}(t^r) - \epsilon^{*}) - p_{l+1}(t^r) \mid t^r \in S^r \} 
= D_{\infty}(p_{l}, p_{l+1}) - \epsilon^{*}.
\end{equation*}
Next consider 
$D_{\infty}(p^{*}_{k-1},p^{*}_{k}) 
= D_{\infty}(p_{k-1},p_{k} + \epsilon^{*} (\11_{A_i^c},-\11))$.
By Lemma~\ref{lem:mutual} again,  we have 
\begin{equation*}
D_{\infty}(p^{*}_{k-1},p^{*}_{k})
= \max \{ 0, p_{k-1}(t^r) - (p_{k}(t^r) - \epsilon^{*}) \mid t^r \in S^r \} 
\leq D_\infty(p_{k-1}, p_{k}) + \epsilon^{*}.
\end{equation*}
Thus
$\sum_{i=k-1}^{l} D_{\infty}(p_{i},p_{i+1})
\geq \sum_{i=k-1}^{l} D_\infty(p^{*}_{i}, p^{*}_{i+1}).$
Consequently we obtain $D_\infty(C^{*}) \leq D_\infty(C)$.

\paragraph{Proof of Lemma~\ref{lem:balanced}.}
It suffices to show (1-2) 
for one particular balanced section $R$.
Let $C = (p_1,p_2,\ldots,p_n)$ be a cycle in $Q_\mu$. 
Since any translation of $R$ 
preserves $D_{\infty}$-distances 
among points in $R$, 
we may assume ($*1$) $(\varphi_{R}(p))^c > p^c$ 
for any point $p$ in $C$.
We use the same idea as above.
For $\alpha \geq 0$, let $\varphi_{R, \alpha}$ be a map defined by
\[
 \varphi_{R, \alpha}(p) 
= p + (\11,-\11) \max \{ \epsilon \in [0,\alpha] 
\mid p^c + \epsilon \11 \leq (\varphi_{R}(p))^c  \}.
\]
Then 
$\varphi_{R} 
= \lim_{\alpha \rightarrow \infty}\varphi_{R, \alpha}$.
Let $i(C)$ be the set of indices $j$ with $\varphi_{R, \alpha}(p) \neq p$ 
for every $\alpha > 0$.
For a small $\epsilon^* > 0$, 
let $C^{*} = (p^{*}_1,p^{*}_2,\ldots, p^{*}_{n}) := 
\varphi_{R, \epsilon^*}(C)$.
It suffices to show $D_{\infty}(C^{*}) \leq D_{\infty}(C)$.
We may assume that there are indices $k,l$ 
such that $k-1,l+1 \not \in i(C)$ and $k,k+1,\ldots, l-1,l \in i(C)$,
where the indices are taken modulo $n$.
For $j = k,k+1, \dots, l-1$,
we have $D_\infty(p_j, p_{j+1}) = D_\infty(p^{*}_j, p^{*}_{j+1})$.
Then
$D_\infty(p^{*}_{l},p^{*}_{l+1}) 
= D_{\infty}(p_{l} + \epsilon^{*} (\11,-\11),p_{l+1}) 
= \max_{s^c \in S^c} \{ p_{l+1}(s^c) -  p_l(s^c) - \epsilon^{*}, 0 \} 
=  \max_{s^c \in S^c} \{ p_{l+1}(s^c) -  p_l(s^c), 0\} - \epsilon^{*}$.
Here the last equality 
follows from $(p_{l+1})^c \not < (p^{*}_{l})^c$ by 
the assumption ($*1$) 
and the balancedness of $R$ (and Lemma~\ref{lem:mutual}~(2)).
Moreover, we have ($*2$)
$
D_{\infty}(p^{*}_{k-1},p^{*}_{k}) 
= D_{\infty}(p_{k-1}, p_{k} + \epsilon^{*} (\11,-\11)) 
 = \max \{p_{k}(s^c) + \epsilon^* - p_{k-1}(s^c), 0\}
\leq \max \{p_{k}(s^c) - p_{k-1}(s^c), 0\} + \epsilon^*$.
Hence we have $D_{\infty}(C^{*}) \leq D_{\infty}(C)$.
Here 
the last inequality in ($*2$) 
holds strictly 
if and only if $p_{k-1}^c > p_k^c$, i.e., $\{p_{k-1}, p_k\}$ is not balanced.
Consequently we have the claim (1-2).

\subsection{Geodesics and embedding}\label{subsec:embedding}
A {\em path} $P$ in $\RR^{S^{cr}}$ is the image of 
a continuous map $\varrho:[0,1] \to \RR^{S^{cr}}$ 
(with respect to the Euclidean topology).
Its {\em length from $\varrho(0)$ to $\varrho(1)$}
is defined by
\[
 \sup 
 \left\{ 
 \sum_{i=0}^{n-1} D_{\infty}(\varrho(t_{i}), \varrho(t_{i+1})) 
 \right\},
\]
where the supreme is taken over 
all $0 = t_0 \leq t_1 \leq t_2 \leq \cdots \leq t_{n} =1$ and $n \geq 1$.
Any segment $[p,q]$  has the length $D_{\infty}(p,q)$ from $p$ to $q$ 
since $D_{\infty}(p,q) = D_{\infty}(p,r) + D_{\infty}(r,p)$ for $r \in [p,q]$.
A subset $R$ in $\RR^{S^{cr}}$ is said 
to be {\em geodesic} if for each pair $p,q \in R$ 
there exists a path in $R$ of length 
$D_{\infty}(p,q)$ from $p$ to $q$.
\begin{Prop}\label{prop:geodesic}
$T_{\mu}$, $Q_{\mu}$, $Q^{+}_{\mu}$, 
and any balanced section $R \subseteq Q_{\mu}$ are all geodesic.
\end{Prop}
\begin{proof}
We show a general statement:
\begin{itemize}
\item[($*$)] For $R \subseteq R' \subseteq \RR^{S^{cr}}$,
suppose that $R'$ is geodesic and there is a cyclically nonexpansive retraction 
$\varphi$ from $R'$ to $R$. Then $R$ is also geodesic.
\end{itemize}
For $p,q \in R$, 
take a path $P$ 
connecting $p$ and $q$ in $R'$
of length $D_{\infty}(p,q)$ from $p$ to $q$.
Suppose that $P$ is the image of $\varrho:[0,1] \to R'$ with 
$(\varrho(0), \varrho(1)) = (p,q)$.
The image $\varphi(P)$ of the cyclically nonexpansive retraction $\varphi$
is also a path in $R$ connecting $p$ and $q$; see Remark~\ref{rem:continuity}.
We show that $\varphi(P)$ has length $D_{\infty}(p,q)$ from $p$ to $q$.
For $n > 0$ and $0 = t_0 \leq t_1 \leq t_2 \leq \cdots \leq t_{n} =1$, 
consider cycle 
$C = (\varrho(t_0), \varrho(t_1),\ldots, \varrho(t_{n}))$.
Then $D_{\infty}(\varphi(C)) \leq D_{\infty}(C)$. 
Since $(p,q) = (\varrho(t_0), \varrho(t_n)) 
= (\varphi \circ \varrho(t_0), \varphi \circ \varrho(t_n))$,
we have 
\begin{equation*}
D_{\infty}(p,q) \leq 
\sum_{i=0}^{n-1} D_{\infty}(\varphi \circ \varrho(t_i),\varphi \circ \varrho(t_{i+1}))
\leq \sum_{i=0}^{n-1} D_{\infty}(\varrho(t_i), \varrho(t_{i+1})) = D_{\infty}(p,q).
\end{equation*}
Thus $\varphi(P)$ has the length $D_{\infty}(p,q)$ from $p$ to $q$, 
and $R$ is geodesic.
Since $P_{\mu}$ is geodesic by convexity, 
Lemmas~\ref{lem:nonexpansive}, \ref{lem:cyclic_nonexpansive}, and 
\ref{lem:balanced} 
imply that $T_{\mu}$, $Q_{\mu}^+$, and any balanced section 
in $Q_{\mu}^+$ are geodesic.

Next we show that an arbitrary balanced section
$R$ in $Q_{\mu}$ is geodesic.
Take any balanced section $R'$ in $Q^{+}_\mu$.
Then the restriction of $\varphi_{R'}$ to $R$
is a bijection, and the inverse map is given by $\varphi_{R}$.
For $p,q \in R$, take a path $P \subseteq R'$
connecting $\varphi_{R'}(p)$
and $\varphi_{R'}(q)$ of length
$D_{\infty}(\varphi_{R'}(p), \varphi_{R'}(q))$
from $\varphi_{R'}(p)$ to $\varphi_{R'}(q)$.
Consider the image $\varphi_R(P)$,
which is a path connecting $p$ and $q$.
By Lemma~\ref{lem:balanced},  
$D_{\infty}(q,p)$ plus the length of $\varphi_R(P)$ from $p$ to $q$
is equal to $D_{\infty}(\varphi_{R'}(q), \varphi_{R'}(p))$ plus
the length $D_{\infty}(\varphi_{R'}(p), \varphi_{R'}(q))$ of $P$ 
from $\varphi_{R'}(p)$ to $\varphi_{R'}(q)$.
Since $D_{\infty}(q,p) + D_{\infty}(p,q) =
D_{\infty}(\varphi_{R'}(q), \varphi_{R'}(p))
+ D_{\infty}(\varphi_{R'}(p), \varphi_{R'}(q))$,
the length of $\varphi_R(P)$ from $p$ to $q$
coincides with $D_{\infty}(p,q)$.
Thus $R$ is geodesic.

Finally, we show that $Q_{\mu}$ is geodesic.
Take a pair $p,q$ in $Q_{\mu}$.
If $\{p,q\}$ is balanced,
then take a balanced section containing $p,q$
and apply the argument above.
So suppose that $\{p,q\}$ is not balanced.
Then we can take a point $r \in Q_{\mu}$
such that $\{p,r\}$ is balanced, $\bar r = \bar q$,
and $D_{\infty}(p,r) + D_{\infty}(r,q) = D_{\infty}(p,q)$.
Take a path $P \subseteq Q_{\mu}$ connecting $p,r$ of
length $D_{\infty}(p,r)$ from $p$ to $r$.
Consider path $P \cup [r,q] \subseteq Q_{\mu}$,
which connects $p$ and $q$ with length $D_{\infty}(p,q)$ from $p$ to $q$.
\end{proof}
Next we show that $\mu$ is realized by $D_{\infty}$-distances 
among a certain subsets in $T_{\mu}$, $Q_{\mu}^+$, 
and any balanced section $R$ in $Q_{\mu}^+$.
Further we denote by 
$T_{\mu,s}$, $Q_{\mu,s}^+$, and $R_{s}$,
the set of points $p$ with $p(s^c) = p(s^r) = 0$
in $T_{\mu}$, $Q^{+}_{\mu}$, and $R$, 
respectively.
For each $s \in S$, we define 
three vectors $\mu_s, \mu_s^{in}, \mu_s^{out} \in \RR^{S^{cr}}$ by
\begin{eqnarray*}
(\mu_s(t^c),\ \mu_s(t^r)) & = & (\mu(t,s), \mu(s,t)), \\
(\mu^{in}_s(t^c),\ \mu^{in}_s(t^r)) & = & 
(\mu(t,s),\ \max_{u \in S} \{ \mu(u,t)- \mu(u,s) \}), \\
(\mu^{out}_s(t^c),\ \mu^{out}_s(t^r)) & = & 
(\max_{u \in S} \{ \mu(t,u) - \mu(s,u) \},\ \mu(s,t)) \quad (t \in S).
\end{eqnarray*}
Both $\mu_s^{in}$ 
and $\mu_s^{out}$ belong to $T_{\mu,s}$. 
In particular 
$\mu_s^{in}$ and $\mu_s^{out}$ 
are unique maximal points in $T_{\mu,s}$
with respect the coordinates in $S^c$ and in $S^r$, respectively.
Now the embedding properties of $T_{\mu}$ and $Q_{\mu}^+$
are summarized as follows:
\begin{Prop}\label{prop:embedding}
The following properties hold:
\begin{itemize}
\item[{\rm (1)}] $T_{\mu,s} = Q_{\mu,s}^{+} = R_{s}$
for each $s \in S$ and any balanced section $R \subseteq Q_\mu^+$. 
\item[{\rm (2)}] 
$(p(s^c),p(s^r)) = (D_{\infty}(T_{\mu,s}, p), D_{\infty}(p, T_{\mu,s})) 
= (D_{\infty}(\mu_s^{out},p), D_{\infty}(p, \mu_s^{in}))$ for
$p \in T_{\mu}$ and $s \in S$. 
\item[{\rm (3)}] 
$\mu(s,t) = D_{\infty}(T_{\mu,s}, T_{\mu,t}) 
=  D_{\infty}(\mu_s^{out}, \mu_t^{in})$ 
for all $s,t \in S$.
\item[{\rm (4)}] If $\mu$ is a directed metric,
then $T_{\mu,s} = \{ \mu_s \}$ for each $s \in S$.
\end{itemize}
\end{Prop}
The property (2) means that the coordinate $p(s^c)$ (resp. $p(s^r)$) is 
the $D_{\infty}$-distance from $T_{\mu,s}$ to 
$p$ (resp. $p$ to $T_{\mu,s}$).
The point $\mu_{s}^{in}$ plays a role of an {\em entrance} of $T_{\mu,s}$;
every point can enter $T_{\mu,s}$ through $\mu_{s}^{in}$.
The point $\mu_{s}^{out}$ is an {\em exit} of $T_{\mu,s}$.
The property (3) means that distance $\mu$ 
can be realized by $D_{\infty}$-distances among 
$T_{\mu,s}$.
By (4), if $\mu$ is a metric, then each $T_{\mu,s}$ is a single point, 
and $(S,\mu)$ is isometrically embedded into 
$(T_{\mu}, D_{\infty})$, $(Q_{\mu}^+, D_{\infty})$, and $(R, D_{\infty})$
by the map $s \mapsto \mu_s$.
\begin{proof}
(1). 
Let $p$ be a point in $T_{\mu,s}$.
By $p(s^c) = p(s^r) = 0$, 
if $p(u) = 0$, 
then $u$ is incident to $s^c$ or $s^r$ in $K(p)$.
Thus $K(p)$ has no isolated vertex, which implies $p \in Q_{\mu}^+$.
Moreover, for any $\epsilon > 0$,
the point $p \pm \epsilon (\11, -\11)$ has a negative coordinate
in $s^c$ or $s^r$. Thus $p \pm \epsilon (\11, -\11) \not \in Q_{\mu}^+$.
This means that every section in $Q_{\mu}^+$ contains $T_{\mu,s}$.

(2). 
Since every $q \in T_{\mu,s}$ satisfies $q(s^c) = 0$,
we have $p(s^c) = p(s^c) -  q(s^c) \leq D_{\infty}(q, p)$.
Thus, it follows that 
$p(s^c) \leq D_{\infty}(T_{\mu,s}, p) \leq D_{\infty}(\mu_{s}^{out}, p)$.
Conversely we have 
$\mu_{s}^{out}(t^r) - p(t^r) = \mu(s,t) - p(t^r) \leq p(s^c)$ 
for all $t^r \in S^r$, which implies
$D_{\infty}(T_{\mu,s}, p) \leq \| ((\mu_{s}^{out})^r - p^r)_{+} \|_{\infty} \leq p(s^c)$ 
(by Lemma~\ref{lem:mutual}~(1)).

(3). Since $p(s^c) = 0$ for all $p \in T_{\mu,s}$ 
and $q(t^c) \geq \mu(s,t)$ for all $q \in T_{\mu,s}$, 
we have $D_{\infty}(T_{\mu,s}, T_{\mu,t}) \geq \mu(s,t)$.
Conversely, by (2), 
$D_{\infty}(T_{\mu,s}, T_{\mu,t}) 
\leq D_{\infty}(\mu_{s}^{out}, \mu_t^{in}) = \mu_s^{in}(t^c) = \mu(s,t)$.

(4). Take $p \in T_{\mu,s}$.
For each $t \in S$, 
we have $p(t^c) \geq \mu(t,s)$ and $p(t^r) \geq \mu(s,t)$.
Since $p(t^c) + p(u^r) = \mu(t,u)$ holds for some $u$, 
we have $p(t^c) = \mu(t,u) - p(u^r) \leq \mu(t,u) - \mu(s,u) \leq \mu(t,s)$, 
where we use the triangle inequality for the last inequality.
Thus $p^c = (\mu_s)^c$, and similarly $p^r = (\mu_s)^r$.
\end{proof}
\begin{Rem}\label{rem:Tmu_s}
$T_{\mu,s}$ is also contractible 
since it is a retract of $\{p \in P_{\mu} \mid p(s^c) = p(s^r) = 0\}$ 
by $\phi$ in Lemma~\ref{lem:nonexpansive}.
\end{Rem}

\subsection{Tight extensions, cyclically tight extensions, and congruence}
\label{subsec:extension}
The goal of this section 
is to show that the tight span and (a fiber of) the tropical polytope
have universal embedding properties 
for metric extensions. 

Let $(S,\mu)$ be a directed metric space.
A directed metric space $(V,d)$ is called 
an {\em extension} of $(S,\mu)$ 
if $S \subseteq V$ and $d(s,t) = \mu(s,t)$ for $s,t \in S$.
An extension $(V,d)$ is said to be {\em tight} 
if there is no extension $(V,d')$ with $d' \leq d$ and $d' \neq d$.
An extension $(V,d)$ is said to be
{\em cyclically tight}
if there is no extension $(V,d')$ of $(S,\mu)$
such that $d'(C) \leq d(C)$ 
for all cycles $C$ in $V$ and $d'(C') < d(C')$ 
for some cycle $C'$.
For convention, 
we also use the notions of extension and tightness 
for metric functions, 
such as ``$d$ is a cyclically tight extension of $\mu$. ''

By Proposition~\ref{prop:embedding}~(3-4),  
$(T_{\mu}, D_{\infty})$, $(Q^+_{\mu}, D_{\infty})$, 
and $(R, D_{\infty})$ for any balanced section $R \subseteq Q^+_{\mu}$ 
are all extensions of $(S,\mu)$.
For an extension $(V,d)$ of $(S,\mu)$ and $x \in V$, 
we define vector $d_x \in \RR^{S^{cr}}$ by
\[
(d_x(s^c), d_x(s^r)) = (d(s,x), d(x,s)) \quad (s \in S).
\]
The main result of this section is the following:
\begin{Thm}\label{thm:extension}
Let $(S,\mu)$ be a directed metric space, and $(V,d)$ its extension.
\begin{itemize}
\item[{\rm (1)}] $d$ is tight if and only if 
there exists an isometric embedding $\rho: V \to T_{\mu}$
such that $\rho(s) = \mu_s$ for $s \in S$.
\item[{\rm (2)}] $d$ is cyclically tight 
if and only if there exists an isometric embedding $\rho: V \to Q_{\mu}^+$
such that $\rho(s) = \mu_s$ for $s \in S$ and $\rho(V)$ is balanced.
\end{itemize}
Moreover, if tight, then such an isometric embedding $\rho$ 
coincides with $x \mapsto d_x$ $(x \in V)$.
\end{Thm}
In particular, (1) says that $(T_{\mu},D_{\infty})$ is 
a unique inclusion-maximal tight extension of $(S,\mu)$.
(2) says that 
$Q^+_{\mu}$ plays a role of the tight span 
with respect to the cyclic tightness.
In contrast to tight extensions,  
there are many maximal cyclically tight extensions, 
which correspond to balanced sections in $Q^+_\mu$.
\begin{proof}
(1). Suppose that $d$ is tight.
Define $\rho: V \to \RR^{S^{cr}}$ by $\rho(x) = d_x$ $(x \in V)$.
Then $\rho(s) = \mu_s \in T_{\mu}$ for $s \in S$.
By  triangle inequality, 
we have $D_{\infty}(\rho(x),\rho(y)) 
= \max \{ \| ({d_y}^c - {d_x}^c)_+\|_{\infty}, 
 \| ({d_x}^r - {d_y}^r)_+\|_{\infty} \} \leq d(x,y)$.
By tightness, we have $D_{\infty}(\rho(x),\rho(y)) = d(x,y)$.
Since $d_x(s) + d_x(t) = d(s,x) + d(x,t) \geq d(s,t) = \mu(s,t)$,
each $\rho(x) = d_x$ belongs to $P_{\mu}$.
Furthermore each $d_x$ belongs to $T_{\mu}$. 
Consider
a nonexpansive retraction $\phi$ in Lemma~\ref{lem:nonexpansive}.
By Proposition~\ref{prop:embedding}~(2),  
$(D_{\infty}(\mu_s, \phi(d_x)), D_{\infty}(\phi(d_x), \mu_s))
= (\phi(d_x)(s^c), \phi(d_x)(s^r)) 
\leq (d(s,x), d(x,s))$. 
By tightness, $\phi(d_x) = d_x$ and thus $d_x \in T_{\mu}$ 
for $x \in V$.

We next show the if part.
Let $\rho:V \to T_{\mu}$ be an isometric embedding.
By Proposition~\ref{prop:embedding}~(2), 
we have $(\rho(x)(s^c), \rho(x)(s^r)) 
= (D_{\infty}(\mu_s, \rho(x)), D_{\infty}(\rho(x), \mu_s)) 
= (d(s,x), d(x,s)) = (d_x(s^c), d_x(s^r))$ for $s \in S$. 
Thus $\rho(x) = d_x$ for $x \in V$.
Take a tight extension $d'$ on $V$ with $d' \leq d$.
Each $d'_x$ belongs to $P_\mu$ as above.
Since $d'_x \leq d_x = \rho(x)$ 
and $\rho(x)$ is minimal in $P_\mu$, 
we have $d'_x = \rho(x)$.
Consequently 
$d'(x,y) = D_{\infty}(d'_x, d'_y) = D_{\infty}(\rho(x), \rho(y)) = d(x,y)$, 
and therefore $d$ is tight.
(2). Suppose that $d$ is cyclically tight.
Since $d$ is tight, 
by (1) it suffices to show
that $\{d_x \mid x \in V\}$ is balanced and is a subset in $Q^+_{\mu}$.
By Lemma~\ref{lem:balanced},  
$\{d_x \mid x \in V\}$ is necessarily balanced.
We verify $d_x \in Q^+_{\mu}$ for each $x \in V$.
Suppose that $K(d_x)$ has an isolated vertex $s^c \in S^c$ (or $s^r \in S^r$).
By the construction of map $\varphi$ in Lemma~\ref{lem:cyclic_nonexpansive},
$D_{\infty}(d_s, d_x) + D_{\infty}(d_x, d_s) (= d_x(s^c) + d_x(s^r))$ 
strictly decreases.
A contradiction to the cyclic tightness of $d$.

We next show the if part.
We show that $d$ is cyclically tight.
Take a cyclically tight extension $d'$ on $V$ 
such that $d'(C) \leq d(C)$ for every cycle in $V$.
By Lemma~\ref{lem:balanced}, 
it suffices to show $\bar d_x = \bar d'_x$ 
for $x \in V$.
Since $d'(s,x) + d'(x,t) + d'(t,s) \leq d(s,x) + d(x,t) + d(t,s)$
for all $s,t \in S$ and $x \in V$,
we have 
$d'_x(s^c) + d'_x(t^r) \leq d_x(s^c) + d_x(t^r)$, 
that is
$d'_x(s^c) - d_x(s^c) \leq d_x(t^r) - d'_x(t^r)$.
In particular, for an edge $s^ct^r$ in $K(d_x)$,
we have 
$\mu(s,t) \leq d'_x(s^c) + d'_x(t^r) \leq d_x(s^c) + d_x(t^r) = \mu(s,t)$,
and hence
$d'_x(s^c) - d_x(s^c) = d_x(t^r) - d'_x(t^r)$.
Therefore, for two edges $s^ct^r$ and $\tilde s^c \tilde t^r$ in
$K(d_x)$, we have ($*$)
$d'_x(s^c) - d_x(s^c) \leq d_x(\tilde t^r) - d'_x(\tilde t^r)
= d'_x(\tilde s^c) - d_x(\tilde s^c) \leq d_x(t^r) - d'_x(t^r)
= d'_x(s^c) - d_x(s^c)$.
Since $K(d_x)$ has no isolated vertex,
it follows from ($*$) that $d_x - d'_x \in (\11,-\11)\RR$.
\end{proof}
Two directed metrics $d,d'$ on the same set $V$
are said to be {\em congruent} if $d(C) = d'(C)$ for each cycle $C$ in $V$.
Then $d$ and $d'$ are congruent 
if and only if there is $\alpha : V \to \RR$ such that
\[
 d(x,y) = d'(x,y) - \alpha(x) + \alpha(y) \quad (x,y \in V).
\]
This is a simple consequence of the well-known fact
in the network flow theory 
that any circulation in a directed network is 
the sum of the incidence vectors of cycles; see \cite{AMO}.

The tropical polytope $\bar Q_\mu$ describes all congruence classes of 
cyclically tight extensions.
This establishes a role of tropical polytopes 
in metric extensions.
For map $\rho: V \to Q_{\mu}$, 
let $\bar \rho: V \to \bar Q_{\mu}$ be the map obtained 
by projecting $\rho(\cdot)$ to $\bar Q_{\mu}$.
\begin{Prop}\label{prop:congruent}
Let $(S,\mu)$ be a directed metric space, and let $V \supseteq S$.
\begin{itemize}
\item[{\rm (1)}] Two cyclically tight extensions $d,d'$ of $\mu$ 
on $V$  are congruent 
if and only if $\bar d_x = \bar d'_x$ for each $x \in V$.
\item[{\rm (2)}] 
A directed metric $d$ on $V$ is congruent to a cyclically tight extension of $\mu$
if and only if 
there exists an isometric embedding $\rho: V \to Q_{\mu}$
such that $\bar \rho(s) = \bar \mu_s$ for $s \in S$ and $\rho(V)$ is balanced.
\end{itemize}
\end{Prop}

\begin{proof}
(1).
The if part 
immediately follows from the previous theorem and Lemma~\ref{lem:balanced}; 
take a balanced section $R$ containing $\{d_x \mid x \in V\}$, 
and apply $\varphi_R$ to $\{d'_x \mid x \in V\}$.
Observe
that $\bar d_x = \bar d'_x$ if and only 
if $d(s,x) - d'(s,x) = d'(x,t) - d(x,t)$ for all $s,t \in S$.
Therefore, $\bar d_x \neq \bar d'_x$ implies 
$d(s,x) + d(x,t) + d(t,s) \neq d'(s,x) + d'(x,t) + d'(t,s)$
for some $s,t \in S$. 
This proves the only-if part.

(2). 
We first show the if part.
Take any balanced section $R$ in $Q^+_{\mu}$ 
(with help of Lemma~\ref{lem:section}).
Consider $\varphi_R \circ \rho$.
The corresponding metric $d'$ on $V$
defined as $d'(x,y) = D_{\infty}( \varphi_R \circ \rho(x), \varphi_R \circ \rho(x))$
is a cyclically tight extension of $\mu$.
By Lemma~\ref{lem:balanced}, $d$ and $d'$ are congruent.
Next we show the only-if part.
Suppose that $d$ is 
congruent to a cyclically tight extension $d'$ of $\mu$.
Then there is $\alpha \in \RR^V$ such that
$d(x,y) - d'(x,y) = - \alpha(x) + \alpha(y)$ 
for $x,y \in V$.
Let $\rho: V \to \RR^{S^{cr}}$ be a map defined by
\[
\rho(x) = d'_x + \alpha(x) (\11,-\11) \quad (x \in V).
\]
Let $d''$ be the metric on $V$ defined as $d''(x,y) = D_{\infty}(\rho(x),\rho(y))$.
We claim $d'' \leq d$.
Indeed,  
we have $D_{\infty}(\rho(x),\rho(y)) 
= \max_{s \in S} ( d'(s,y) + \alpha(y) - d'(s,x) - \alpha(x) )_+
\leq d'(x,y) + \alpha(y) - \alpha(x) = d(x,y)$ (by triangle inequality).
If $\rho(V)$ is not balanced, 
then there is a cycle $C$ in $V$ 
such that $d'(C) < d''(C) \leq d(C)$.
This is a contradiction to $d'(C) = d(C)$.
Thus $\rho(V)$ is balanced and $d = d''$ holds.
\end{proof}

\subsection{Dimension criteria}\label{subsec:dimension}
Here we present matching-type criteria for 
dimensions of
tight span $T_{\mu}$ and tropical polytope $\bar Q_\mu$. 
For tropical polytopes, such a criterion 
has already been given by \cite{DSS05}.

For two $n$-element subsets $A^c \subseteq S^c$
and $B^r \subseteq S^r$,
let $K_{A^c,B^r}$ be the complete bipartite graph 
with the bipartition $(A^c,B^r)$.
A {\em matching} is a subset of edges 
which have no common vertices.
A matching $M$ is said to {\em perfect} if all vertices are covered by $M$.
For a matching $M$, let $\mu(M) = \sum_{a^cb^r \in M} \mu(a,b)$.
We consider the following maximum matching problems:
\begin{description}
\item[MT: ] Maximize $\mu(M)$ over all matchings $M$ in $K_{A^c,B^r}$.  
\item[PMT: ] Maximize $\mu(M)$ over all perfect matchings $M$ in $K_{A^c,B^r}$.  
\end{description}
Since $\mu$ is nonnegative, 
the maximum values of two problems are same.
Our interest is how the maximum is attained.
\begin{Thm}
For a positive integer $n$, 
the following two conditions are equivalent:
\begin{itemize}
\item[{\rm (a)}] $\dim T_{\mu} \geq n$.
\item[{\rm (b)}] 
 There exist two $n$-element subsets $A^c \subseteq S^c$
 and $B^r \subseteq S^r$
 such that the maximum of MT
is uniquely attained by a perfect matching $M$. 
\end{itemize}
\label{thm:dim}
\end{Thm}
\begin{Thm}[{\cite[Theorem 4.2]{DSS05}}]\label{thm:trop_dim}
For a positive integer $n$, 
the following two conditions are equivalent:
\begin{itemize}
\item[{\rm (a)}] $\dim \bar Q_\mu \geq n - 1$.
\item[{\rm (b)}] 
 There exist two $n$-element subsets $A^c \subseteq S^c$
 and $B^r \subseteq S^r$
 such that the maximum of PMT is uniquely attained.
\end{itemize}
\end{Thm}
In \cite{DSS05, DS04}, 
the maximum integer $n$ with property (b) is called 
the {\em tropical rank} of (the distance matrix of) $- \mu$.
We give a proof sketch of Theorem~\ref{thm:dim};
this proof method has already been established 
in \cite{DSS05} and \cite[Appendix]{Hirai06AC}. 
\begin{proof}[Sketch of the proof of Theorem~\ref{thm:dim}]
(a) $\Rightarrow$ (b). Since $\dim T_{\mu} \geq n$, there exists 
$p \in T_{\mu}$ such that $K(p)$ has
$n$ components having no vertex $s$ with $p(s) = 0$ (Lemma~\ref{lem:dim}).
Let $M$ be a set of $n$ edges obtained 
from each of the $n$ components,
and let $A^c \subseteq S^c$ and $B^r \subseteq S^r$ 
be the sets of vertices of $M$.
By construction, $M$ is a perfect matching
of $K_{A^c,B^r}$.
Then $\mu(M) = 
\sum_{a \in A^c} p(a) + \sum_{b \in B^c} p(b)$.
Take any matching $M'$ in $K_{A^c,B^r}$.
Then 
$\mu(M') \leq \sum_{a \in A^c} p(a) + \sum_{b \in B^c} p(b)$
by $p(a) + p(b) \geq \mu(a,b)$.
So the equality is  attained exactly when $M = M'$.

(b) $\Rightarrow$ (a). 
With help of the bipartite matching theory~\cite[Part II]{SchrijverBook} and 
the strict complementary slackness theorem in linear programming, 
there is a positive vector $p^*: A^c \cup B^r \to \RR_{+}$ 
such that $p^*(a^c) + p^*(b^r) = \mu(a,b)$ for $a^cb^r \in M$ 
and $p^*(a^c) + p^*(b^r) > \mu(a,b)$ otherwise.
Therefore, by the same argument in Lemma~\ref{lem:dim}, 
the set $T_{\mu; A^c,B^r}$ 
of minimal elements of polyhedron $\{p \in \RR^{A^c \cup B^r}_+ 
\mid p(a^c) + p(b^r) \geq \mu(a,b) \  (a,b) \in A^c \times B^r \}$
has dimension at least $n$.
Any $p \in T_{\mu; A^c,B^r}$ can be
extended to $p: S^{cr}  \to \RR_{+}$ so that 
$p \in P_{\mu}$. 
Then decrease $p(u)$ for $u \not \in A^c \cup B^r$ so that $p \in T_{\mu}$.
Therefore the projection of $T_{\mu}$ to $\RR^{A^c \cup B^r}$
includes $T_{\mu; A^c,B^r}$. 
This implies $\dim T_{\mu} \geq \dim T_{\mu; A^c,B^r} \geq n$. 
\end{proof}

\section{Combinatorial characterizations of distances\\ with dimension one}
\label{sec:dim1}
A {\em tree metric} is a metric that can be represented as
the graph metric of some tree.
The {\em tree metric theorem}~\cite{Buneman74, Pereira69, Zareckii65} 
says that an undirected metric $\mu$ on a set $S$
is a tree metric if and only if it satisfies the
{\em four point condition}:
\[
\mu(s,t) + \mu(u,v) \leq \max \{\mu(s,u)+ \mu(t,v), \mu(s,v) + \mu(t,u)\}
\quad (s,t,u,v \in S).
\]
See also \cite[Chapter 7]{SempleSteel}.
Dress~\cite{Dress84} interpreted the four point condition 
as a criterion for one-dimensionality of 
the undirected tight span of $\mu$, 
and derived the tree metric theorem 
from the embedding property; 
see also \cite{Hirai06AC}.

In this section, 
we apply this idea to our directed versions of tight spans, 
and derive combinatorial characterizations of classes of distances
realized by oriented trees.
An {\em oriented tree} 
is a directed graph whose underlying undirected graph is a tree.
For an oriented tree $\mit\Gamma$
and a nonnegative edge-length $\alpha: E\mit\Gamma \to \RR_{+}$, 
we define a directed metric 
$D_{\mit\Gamma, \alpha}$ on $V\mit\Gamma$ as follows.
For two vertices $x,y \in V\mit\Gamma$, 
let $P[x,y]$ denote the set of edges 
forming a unique path connecting $x,y$
in the underlying undirected tree, 
and let $\vec P[x,y]$ be the set of edges in 
$P[x,y]$ whose direction is the same as 
the direction from $x$ to $y$; in particular 
$P[x,y] = \vec P[x,y] \cup \vec P[y,x]$ (disjoint union).
Then let $D_{\mit\Gamma, \alpha}(x,y) 
= \sum \{ \alpha(e) \mid e \in \vec P[x,y] \}$.
For a distance $\mu$ on $S$, an {\em oriented-tree realization} 
$({\mit\Gamma}, \alpha; \{F_s\}_{s \in S})$
consists of an oriented tree ${\mit\Gamma}$, 
a positive edge-length $\alpha: E{\mit\Gamma} \to \RR_{+}$, 
and a family $\{F_s \mid s \in S\}$ of subtrees 
in ${\mit\Gamma}$
such that 
\[
\mu(s,t) = D_{\mit\Gamma, \alpha}(F_s, F_t) \quad (s,t \in S),
\]
where a {\em subtree}
is a subgraph whose underlying undirected graph is connected, 
and $D_{\mit\Gamma, \alpha}(F_s, F_t)$ denotes 
the shortest distance from $F_s$ to $F_t$.
A {\em directed path} 
is an oriented tree each of whose vertices has at most one leaving edge 
and at most one entering edge.
The main results in this section are the following:
\begin{Thm}\label{thm:path}
For a directed distance $\mu$ on $S$, 
the following conditions are equivalent:
\begin{itemize}
\item[{\rm (a)}] $\mu$ has an oriented-tree realization $(\mit\Gamma, \alpha; \{F_s\}_{s \in S})$ 
                 so that ${\mit\Gamma}$ is a directed path.
\item[{\rm (b)}] $\dim T_{\mu} \leq 1$.
\item[{\rm (c)}] For $s,t,u,v \in S$ (not necessarily distinct), we have
\begin{equation*}
\mu(s,u) + \mu(t,v) \leq \max \{ \mu(s,v) + \mu(t,u),\ \mu(s,u),\ \mu(s,v),\ \mu(t,u),\ \mu(t,v)  \}.
\end{equation*}
\end{itemize}
\end{Thm}
\begin{Thm}\label{thm:tree}
For a directed distance $\mu$ on $S$, 
the following conditions are equivalent:
\begin{itemize}
\item[{\rm (a)}] $\mu$ has an oriented-tree realization 
$(\mit\Gamma, \alpha; \{F_s\}_{s \in S})$ 
so that each subtree $F_s$ is a directed path.
\item[{\rm (b)}] $\dim \bar Q_\mu \leq 1$. 
\item[{\rm (c)}] For $x,y,z,u,v,w \in S$ (not necessarily distinct), we have
\begin{eqnarray*}
&& \mu(x,u) + \mu(y,v) + \mu(z, w) \\[0.2em]
&& \leq 
\max \left\{   
\begin{array}{cc}
 \mu(x,v) + \mu(y,u) + \mu(z, w), & \mu(x,v) + \mu(y,w) + \mu(z, u), \\
 \mu(x,w) + \mu(y,u) + \mu(z, v), & \mu(x,w) + \mu(y,v) + \mu(z, u), \\
\mu(x,u) + \mu(y,w) + \mu(z, v)
\end{array}
\right\}, \\
\end{eqnarray*}
i.e., the tropical rank of $-\mu$ is at most $2$. 
\end{itemize}
\end{Thm}
\begin{proof}[Proof of Theorems \ref{thm:path} and \ref{thm:tree}]
In both theorems, 
(b) $\Leftrightarrow$ (c) follows from dimension criteria 
(Theorems~\ref{thm:dim} and \ref{thm:trop_dim}).
We prove (b) $\Rightarrow$ (a) 
from embedding property (Proposition~\ref{prop:embedding}).
Suppose first $\dim T_{\mu} \leq 1$. 
Then $T_{\mu}$ is a tree since $T_{\mu}$ is contractible 
(Remark~\ref{rem:continuity}).
Recall (\ref{eqn:coordinate}).
For each face (segment) $F = [p,q]$ in $T_{\mu}$, 
there are $A,B \subseteq S$
such that either $q - p = D_{\infty}(p,q) (\11_{A^c}, - \11_{B^r})$
or  $p - q = D_{\infty}(q,p) (\11_{A^c}, -\11_{B^r})$.
Therefore subspace $([p, q], D_{\infty})$
is isometric to the segment in $(\RR, D_{\infty}^+)$.
Then an oriented-tree realization $({\mit\Gamma}, \alpha; \{F_s\}_{s \in S})$ 
is obtained by the following way. 
The underlying undirected tree of ${\mit\Gamma}$ 
is the $1$-skeleton graph of $T_{\mu}$.
We orient each undirected edge $pq$ as
$p \rightarrow q$
when $D_{\infty}(p,q) > 0$ (and $D_{\infty}(q,p) = 0$).
The edge-length $\alpha(pq)$ 
is given by $\max \{ D_{\infty}(p,q), D_{\infty}(q,p)\}$.
Consider $T_{\mu,s}$, which is also the union of segments and 
is contractible (Remark~\ref{rem:Tmu_s}).
Let $F_s$ be its corresponding subgraph in ${\mit\Gamma}$, 
which is a subtree.
By Proposition~\ref{prop:geodesic},
the restriction of $(T_{\mu}, D_{\infty})$ 
to the set of the vertices of $T_{\mu}$
is isometric to $(V{\mit\Gamma}, D_{{\mit\Gamma}, \alpha})$.
By Proposition~\ref{prop:embedding}~(3), 
$\mu$ is realized by $({\mit\Gamma},\alpha; \{F_s\}_{s \in S})$.
Next we verify that $\mit\Gamma$ is a directed path.
Suppose to the contrary that there is a vertex having two entering edges 
or two leaving edges.
Then there exists $s,t \in S$ such that
the unique path joining $F_s$ and $F_t$ passes through these two edges. 
Then both $\mu(s,t)$ and $\mu(t,s)$ are positive.
This contradicts (c) 
for the case $(v,u) = (s,t)$.

Next suppose $\dim \bar Q_{\mu} \leq 1$. 
Consider the following section $R \subseteq Q_{\mu}^+$:
\[
 R = \{ p \in Q_{\mu}^+ \mid \forall \epsilon > 0, p + \epsilon (\11,-\11) \not \in Q_{\mu}^{+} \}.
\]
Then $R$ is balanced since it consists of points $p$ such 
that $p^r$ has a zero component.
Also $R$ is a subcomplex (the union of faces) of $T_{\mu}$. 
Therefore $R$ is also a tree each of whose 
segments is isometric to 
a segment in $(\RR, D_{\infty}^+)$.
Similarly, from the $1$-skeleton of $R$
we obtain an oriented-tree realization 
$({\mit\Gamma}, \alpha; \{F_s \}_{s \in S})$ of $\mu$.
We verify that each subtree $F_s$ is a directed path.
Suppose to the contrary
that there is a vertex $v$ in $F_s$
such that $v$ has two entering edges or two leaving edges in $F_s$.
Then we can take a point $p$ in $T_{\mu,s}$ 
such that $p(s^r)= D_{\infty}(p,T_{\mu,s}) = D_{\infty}(p, \mu_s^{in}) > 0$ or 
$p(s^c)= D_{\infty}(T_{\mu,s},p) = D_{\infty}(\mu_s^{out},p) > 0$, 
which contradicts $(p(s^c), p(s^r)) = (0,0)$ 
(Proposition~\ref{prop:embedding}~(2)).

Next we show (a) $\Rightarrow$ (b) or (c).
Suppose that $\mu$ is realized by $({\mit\Gamma}, \alpha; \{F_s\}_{s \in S})$.
It suffices to consider the case where
each $F_s$ is a single vertex.
Indeed, suppose that $F_s$ is 
a directed path of tail $v^-$ and head $v^+$. 
Let $S' = S \setminus s \cup \{s^-,s^+\}$, and let 
$F_{s^-}$ and $F_{s^+}$ be subtrees consisting of 
singletons $v^-$ and $v^+$, respectively.
Consider the new distance $\mu'$ on $S'$ 
by $\mu'(t,u) = D_{\mit\Gamma, \alpha}(F_t,F_u)$ for $t,u \in S'$.
Then the distance matrix of $\mu$ is 
obtained by deleting 
the $s^-$-th row and the $s^+$-th column from $\mu'$.
Namely $\mu$ is a submatrix of $\mu'$.
By dimension criteria, 
$\dim T_{\mu} \leq \dim T_{\mu'}$ and 
$\dim \bar Q_{\mu} \leq \dim \bar Q_{\mu'}$.
So suppose that each $F_s$ is a single vertex; 
in particular $\mu$ is a metric.
We first show (a) $\Rightarrow$ (c) in the first theorem.
Now $\mu$ can be regarded as $D^{+}_{\infty}$-distances among points $x_s$ 
in the real line $\RR$. 
Namely $\mu(s,t) = (x_t - x_s)_+$ for $s,t \in S$.
Take any $s,t,u,v \in S$.
We verify the condition in (c).
If $x_u \leq x_s$, 
then $\mu(s,u) = 0$ and $\mu(s,u) + \mu(t,v) = \mu(t,v)$.
So we may assume that $x_u > x_s$ and $x_v > x_t$.
It suffices to consider 
three cases (i) $x_s < x_u \leq x_t < x_v$, 
(ii)  $x_s \leq x_t < x_u \leq x_v$, and
(iii)  $x_s \leq x_t < x_v \leq x_u$.
(i) implies $\mu(s,u) + \mu(t,v) \leq \mu(s,v)$.
Both (ii) and (iii) imply $\mu(s,u) + \mu(t,v) = \mu(s,v) + \mu(t,u)$.

Finally
we show (a) $\Rightarrow$ (b) in the second theorem; 
we use an alternative approach to avoid case disjunctions.
We claim:
\begin{itemize}
\item[($*$)] $\mu$ is congruent to some tree metric $d$. 
\end{itemize}
Suppose that ($*$) is true.
Then there is $\alpha: S \to \RR$ such that
$\mu(s,t) = d(s,t) - \alpha(s) + \alpha(t)$ $(s,t \in S)$.
One can easily see that $Q_d$ is a translation of $Q_\mu$.
It is known that 
the dimension of the tropical polytope spanned 
by a tree metric
is at most $1$~\cite[Theorem 28]{DS04}.
Therefore $\dim \bar Q_\mu$ is also at most $1$.

We show the claim ($*$).
Now $\mu$ is realized by $({\mit\Gamma}, \alpha; \{F_s\}_{s \in S})$
with each subtree $F_s$ being a single vertex $v_s$.
For an edge $e$ in ${\mit\Gamma}$, 
let $(A_e,B_e)$ be the ordered bipartition of $S$
such that $s$ belongs to $A_e$ if and only 
if $v_s$ belongs to the connected component of ${\mit\Gamma} \setminus e$
containing the tail of $e$.
Let $\vec \delta_{A_e,B_e}$ be a directed metric on $S$ defined by
\[
 \vec \delta_{A_e,B_e}(s,t) = 
\left\{
\begin{array}{ll}
1 & {\rm if}\ (s,t) \in A_e \times B_e, \\
0 & {\rm otherwise},
\end{array}
\right. \quad (s,t \in S).
\]
By construction, we have
$
\mu = \sum_{e \in E\mit\Gamma} \alpha(e) \vec \delta_{A_e,B_e}.
$
Then one can easily see that $\vec \delta_{A_e,B_e}$ is 
congruent to  $(\vec \delta_{A_e,B_e} + \vec \delta_{B_e,A_e})/2$.
Here 
$\vec \delta_{A_e,B_e} + \vec \delta_{B_e,A_e}$ coincides with
the {\em split metric} of {\em split} (bipartition) $\{A_e,B_e\}$.
So $\mu$ is congruent to a nonnegative 
sum of split metrics for the pairwise {\em compatible} family
of splits
$\{ \{A_e,B_e\} \mid e \in E\mit\Gamma \}$, 
which is just a tree metric; see \cite{SempleSteel}.
\end{proof}
A directed metric $\mu$ is called 
a {\em directed tree metric} if it has
an oriented-tree realization $({\mit\Gamma}, \alpha; \{F_s\}_{s \in S})$
such that each subtree $F_s$ is a single vertex.
For a directed metric $\mu$, 
let $\mu^t$ denote the metric obtained by transposing $\mu$, 
i.e., $\mu^t(x,y) := \mu(y,x)$.
The following corollary sharpens~\cite[Theorem 28]{DS04}, and 
includes a nonnegative version of~\cite[Theorem 5]{Patrinos72}.
\begin{Cor}\label{cor:treemetric}
For a directed metric $\mu$ on $S$, the following conditions are equivalent:
\begin{itemize}
\item[{\rm (a)}] $\mu$ is a directed tree metric.
\item[{\rm (b)}] $\mu$ is congruent to a tree metric. 
\item[{\rm (c)}] $\mu + \mu^{t}$ satisfies the four point condition
and $\mu(x,y) + \mu(y,z) + \mu(z,x) = \mu(z,y) + \mu(y,x) + \mu(x,z)$ 
holds for every triple $x,y,z \in S$.
\item[{\rm (d)}] the tropical rank of $-\mu$ is at most $2$.
\end{itemize}
\end{Cor}
\begin{proof}
In the proof of (b) $\Rightarrow$ (a) in the previous theorems, 
if $\mu$ is a metric, then we can take $F_s$ as a single vertex 
since $T_{\mu,s}$ is a single point $\mu_s$ 
(Proposition~\ref{prop:embedding}~(4)).
So we obtain
(a) $\Rightarrow$ (b) $\Rightarrow$ (d) $\Rightarrow$ (a) 
from the proof above.
We verify (b) $\Leftrightarrow$ (c).
Decompose $\mu$ as $(\mu + \mu^{t})/2 + (\mu - \mu^{t})/2$.
Then the second condition in (c) is rephrased as 
$(\mu - \mu^{t})(C) = 0$ for every cycle $C$ consisting of three elements.
One can easily see that this is equivalent 
to $(\mu - \mu^{t})(C) = 0$ for every cycle $C$, i.e., 
$\mu$ is congruent to $(\mu + \mu^{t})/2$.
From this fact and the tree metric theorem, 
we obtain (b) $\Leftrightarrow$ (c).
\end{proof}

\section{Multicommodity flows}\label{sec:multiflows}
Originally 
our directed version of tight spans
was motivated by the multicommodity flow theory 
in combinatorial optimization; 
see \cite[Section 17]{AMO} and \cite[Part VII]{SchrijverBook}.
In this section, we briefly sketch 
how tight spans and tropical polytopes 
are applied to the study of multicommodity flows in directed networks;
see \cite{Hirai09, Kar98EJC} 
for applications of undirected tight spans to 
multicommodity flows in undirected networks.

By a {\em network} 
we mean a quadruple
$(V,E,S,c)$ consisting of 
a directed graph $(V,E)$, a specified vertex subset $S \subseteq V$, 
and a nonnegative integer-valued edge-capacity $c: E \to \ZZ_{+}$.
We call a vertex in $S$ a {\em terminal}.
A directed path $P$ in $(V,E)$ is called an {\em $S$-path} 
if $P$ connects distinct terminals in $S$.
A {\em multiflow} ({\em multicommodity flow})
is a pair $({\cal P}, \lambda)$ 
of a set ${\cal P}$ of $S$-paths and 
a nonnegative flow-value function $\lambda:{\cal P} \to \RR_{+}$
satisfying the capacity constraint:
\[
\sum \{ \lambda(P) \mid  P \in {\cal P}, \mbox{$P$ contains $e$}\} \leq c(e) \quad (e \in E).
\]
Let $\mu: S \times S \to \RR_{+}$ be a nonnegative 
weight defined on the set of all ordered pairs on terminals.
For a multiflow $f = ({\cal P},\lambda)$, 
its {\em flow-value} $\val (\mu, f)$ is defined 
by $\sum_{P \in {\cal P}} \mu(s_P,t_P) \lambda(P)$, 
where $s_P$ and $t_P$ denotes the starting vertex and the end vertex of $P$, 
respectively.
Namely $\mu(s,t)$ represents a value of a unit $(s,t)$-flow.
The $\mu$-weighted maximum multiflow problem 
is formulated as follows.
\begin{description}
\item[MCF: ] \quad Maximize $\val (\mu,f)$
 over all multiflows $f$ in $(V,E,S,c)$
\end{description}
This is one of fundamental problems in operations research, 
and has a wide range of practical applications; see~\cite[Section 17.1]{AMO}.

We regard $\mu$ as a directed distance on $S$.
For a simplicity, we assume that {\em $\mu$ is a metric}.
As is well-known in the multiflow theory \cite{Lomonosov85},
an LP-dual to (MCF) is a linear optimization over metrics:
\begin{eqnarray}\label{eqn:LP-dual}
\mbox{Minimize} && \sum_{xy \in E} c(xy) d(x,y) \\
\mbox{subject to} && \mbox{$d$ is a directed metric on $V$}, \nonumber \\
 && d(s,t) = \mu(s,t) \quad (s,t \in S). \nonumber
\end{eqnarray}
Here recall notions in Section~\ref{subsec:extension}.
This is nothing but 
a linear optimization over 
all extensions of $(S,\mu)$.
{\em Since $c$ is nonnegative, 
the minimum is always attained by a tight extension.}
By Theorem~\ref{thm:extension}~(1), 
possible candidates $d$ for optimum 
are isometrically embedded into 
$(T_{\mu},D_{\infty})$.
So the problem reduces to be 
an optimization over isometric embeddings $\rho$.
Thus we have the following min-max relation sharpening the LP-duality.
\begin{Thm}\label{thm:Tdual}
Let $(V,E,S,c)$ be a network and let $\mu$ be a directed metric on $S$.
The following min-max relation holds:
\begin{eqnarray}\label{eqn:min-max}
&& \max \{\; \val(\mu,f) \mid \mbox{$f$: multiflow in $(V,E,S,c)$} \} \\[0.5em]
&& = \min \left\{  \sum_{xy \in E} c(xy)D_{\infty}(\rho(x), \rho(y)) 
\bigmid  \rho: V \to T_\mu,\ \rho(s) = \mu_s\ (s \in S) \right\}. \nonumber
\end{eqnarray}
\end{Thm}
We give two interpretations of this minimization problem in 
RHS of (\ref{eqn:min-max}) below.
The first is a {\em facility location} on $T_{\mu}$.
There are facilities at points $\mu_s$ $(s \in S)$ in $T_{\mu}$.
We are going to locate new facilities $x \in V \setminus S$ 
at points $\rho(x)$ in $T_{\mu}$.
Here facilities communicate with each other, and
have the {\em communication cost}, 
which is a monotone function of 
their distances $D_{\infty}(\rho(x),\rho(y))$.
The objective is to find a location of the minimum communication cost.
In the literature of the location theory~\cite{TFL},  
the undirected version of this problem is known as 
{\em multifacility location problem}.
The second interpretation comes from electrical circuits.
We associate each vertex $x$ 
with a $T_{\mu}$-valued {\em potential} $\rho(x)$.
The objective is 
to minimize the {\em energy}, a function of
{\em potential differences} $D_{\infty}(\rho(x),\rho(y))$ 
measured in $(T_{\mu}, D_{\infty})$, 
under the {\em boundary condition} $\rho(s) = \mu_s$.

Suppose the case where network $(V,E,S,c)$ is {\em Eulerian},
i.e., for each vertex $x$, 
the sum of capacities over all edges entering $x$
is equal to the sum of capacities over all edges leaving $x$.
In this case, capacity function $c$ is decomposed 
into the characteristic vectors of 
(possibly repeated) cycles $C_1,C_2,\ldots,C_m$ in $(V,E)$.
For any metric $d$ on $V$, the objective value in (\ref{eqn:LP-dual}) 
is given by
\[
d(C_1) + d(C_2) + \cdots + d(C_m).
\]
Therefore, {\em the minimum of $(\ref{eqn:LP-dual})$
is always attained by a cyclically tight extension.}
Moreover two congruent metrics have the same objective value.
By Theorem~\ref{thm:extension}~(2)
and Proposition~\ref{prop:congruent}~(2), 
any metric $d$ congruent to a cyclically tight extension
is isometrically embedded into $(Q_{\mu}, D_{\infty})$.
Thus we have the following min-max relation: 
\begin{Thm}\label{thm:Qdual}
Let $(V,E,S,c)$ be an Eulerian network and let $\mu$ be a directed metric on $S$.
Then the following min-max relation holds:
\begin{eqnarray}\label{eqn:min-max_Eulerian}
&& \max \{\; \val(\mu,f) \mid \mbox{$f$: multiflow in $(V,E,S,c)$} \} \\[0.5em]
&& = \min \left\{  \sum_{xy \in E} c(xy)D_{\infty}(\rho(x), \rho(y))
\bigmid  \rho: V \to Q_{\mu},\ \bar \rho(s) = \bar \mu_s\ (s \in S) \right\}. \nonumber
\end{eqnarray}
\end{Thm}
Here we can take $\rho$
so that $\rho(V)$ 
lies on any fixed balanced section $R \subseteq Q_{\mu}$.
So the RHS in (\ref{eqn:min-max_Eulerian}) 
is essentially a facility location problem 
on the tropical polytope $\bar Q_{\mu}$. 

In the single commodity maximum flow problem, 
a classic theorem by Ford-Fulkerson says 
that the maximum flow value is equal to the minimum value of cut capacity, 
and there always exists an {\em integral} maximum flow, i.e.,  
a maximum flow whose flow-value function is integer-valued; 
see \cite{AMO, SchrijverBook}.
Such a {\em combinatorial min-max theorem} and an {\em integrality theorem}
are closely related to the geometry of tight spans and tropical polytopes. 
This issue, however, is beyond the scope of the paper.
So we will discuss it in the next paper~\cite{HK}.

\section*{Acknowledgments}
The first author is supported by 
a Grant-in-Aid for Scientific Research 
from the Ministry of Education, Culture, Sports,
Science and Technology of Japan.
The second author is supported by 
Nanzan University Pache Research Subsidy I-A-2
for the 2008 academic year.


\begin{thebibliography}{40}
\small

\bibitem{AMO}
R. K. Ahuja, T. L. Magnanti, and J. B.  Orlin,  
{\it Network Flows---Theory, Algorithms, and Applications,} 
Prentice Hall,  Englewood Cliffs, 1993. 





\bibitem{Buneman74}
P. Buneman, 
A note on the metric properties of trees,  
{\em Journal of Combinatorial Theory, Series B} {\bf 17} (1974), 48--50.
%

\bibitem{CL94}
 M. Chrobak and L. L. Larmore, Generosity helps or an 11-competitive algorithm for
 three servers, {\it Journal of Algorithms} {\bf 16} (1994), 234--263.


\bibitem{DSS05}
M. Develin, F. Santos, and B. Sturmfels, 
On the rank of a tropical matrix, 
In: J. E. Goodman, J. Pach and E. Welzl (eds.), 
{\em Combinatorial and Computational Geometry},  
Mathematical Sciences Research Institute Publications, 52. 
Cambridge University Press, Cambridge, 2005, 213--242.


\bibitem{DS04}
M. Develin and B. Sturmfels, Tropical convexity, 
{\it Documenta Mathematica} {\bf 9} (2004), 1--27.



\bibitem{Dress84}
 A. W. M. Dress, Trees, tight extensions of metric spaces, 
and the cohomological dimension of certain groups: 
a note on combinatorial properties of metric spaces,
 {\it Advances in Mathematics} {\bf 53} (1984), 321--402.




\bibitem{Hirai06AC}
  H. Hirai, Characterization of the distance between subtrees of a tree
  by the associated tight span, {\it Annals of Combinatorics} {\bf 10} (2006),
  111--128.

\bibitem{Hirai09}
 H. Hirai, Tight spans of distances and the dual fractionality
 of undirected multiflow problems,
 {\it Journal of Combinatorial Theory, Series B} {\bf 99} (2009), 843--868.


\bibitem{HK}
H. Hirai and S. Koichi, in preparation.


\bibitem{Isbell64}
 J. R. Isbell, Six theorems about injective metric spaces,
 {\it Commentarii Mathematici Helvetici} {\bf 39} (1964), 65--76.




\bibitem{Kar98EJC}
A. V. Karzanov, Minimum 0-extensions of graph metrics,
{\it European Journal of Combinatorics} {\bf 19} (1998), 71--101.




 \bibitem{Lomonosov85}
  M. V. Lomonosov, Combinatorial approaches to multiflow problems,
 {\it Discrete Applied Mathematics} {\bf 11} (1985), 93 pp.


\bibitem{Patrinos72}
  A. N. Patrinos and S. L. Hakimi,
  The distance matrix of a graph and its tree realization,
  {\it Quarterly of Applied Mathematics} {\bf 30} (1972), 255--269.


\bibitem{SchrijverBook}
 A. Schrijver, 
 {\it Combinatorial Optimization ---Polyhedra and Efficiency}, 
 Springer-Verlag, Berlin, 2003.

\bibitem{SempleSteel}
C. Semple and M. Steel, 
{\em Phylogenetics,} 
Oxford University Press, Oxford, 2003. 


\bibitem{Pereira69}
J. M. S. Sim{\~o}es-Pereira, 
A note on the tree realizability of a distance matrix, 
{\em Journal of Combinatorial Theory} {\bf 6} (1969), 303--310. 


\bibitem{TFL}
B. C. Tansel, R. L.  Francis, and T. J.  Lowe,  
Location on networks: a survey. I-II.  
{\em Management Sciences} {\bf  29}  (1983), 482--511. 




\bibitem{Zareckii65}
K. A. Zarecki{\u\i},
Constructing a tree on the basis 
of a set of distances between the hanging vertices, 
{\em Uspekhi Matematicheskikh Nauk} {\bf 20} (1965), 90--92. (in Russian)


\end{thebibliography}
\end{document}